\documentclass[11pt]{article}
\usepackage[mathvars]{brsheader} 
\abbreviateBBMathFont

\newcommand{\xto}[1]{\xrightarrow{#1}}
\newcommand{\MRto}{\xrightarrow[\text{MR}]{}}

\title{Cuboids are canonically Ramsey}
\author{Benedict \textsc{Randall Shaw}\footnotemark[1]}
\date{February 2026}

\begin{document}

\maketitle

\renewcommand{\thefootnote}{\fnsymbol{footnote}}
\footnotetext[1]{\href{mailto:bwr26@cam.ac.uk}{bwr26@cam.ac.uk}, Department of Pure Mathematics and Mathematical Statistics (DPMMS), University of Cambridge, Wilberforce Road, Cambridge, CB3 0WA, United Kingdom}

\begin{abstract}
    We say a set of points \(C\subset \BBR^n\) is \emph{canonically Ramsey} if there is some set of points \(S\subset \BBR^{n'}\) such that any colouring of \(S\), with any number of colours, admits either a monochromatic or rainbow copy of \(C\)---that is to say, some set of points congruent to \(C\) either all receive the same colour, or all receive different colours. Mao, Ozeki, and Wang \cite{mao2022euclideangallairamseytheory} introduced this notion, proving that 30-60-90 triangles are canonically Ramsey, since when various other canonically Ramsey configurations have been identified (by Gehér, Sagdeev, and Tóth \cite{geher2026}, and others). Fang, Ge, Shu, Xu, Xu, and Yang \cite{fang2025canonicalramseytrianglesrectangles} showed that all triangles and rectangles are canonically Ramsey, and asked whether all cuboids are canonically Ramsey. Here cuboids are sets of the form \(\{0,b_1\}\times\dots\times\{0,b_s\}\), and in particular may have dimension greater than three. We resolve this question, proving that all cuboids are canonically Ramsey.
\end{abstract}

\section{Introduction}
For positive integers \(n,r\), and some configuration (finite set) \(C\subset \BBR^n\), we write
\[\BBR^n \xto{r} C\]
if any \(r\)-colouring of the points of \(\BBR^n\) admits a monochromatic copy of \(C\)---that is to say, given any \(r\)-colouring of the points of \(\BBR^n\), there is some set of points congruent to \(C\) all of which receive the same colour.
More generally, for any \(S\subset\BBR^n\), we write
\[S \xto{r} C\]
if any \(r\)-colouring of the points of \(S\) admits a monochromatic copy of \(C\). A compactness argument shows that \(\BBR^n \xto{r} C\) if and only if some finite \(S\subset \BBR^n\) has \(S\xto{r} C\).
A configuration is called \emph{Ramsey} if, for any such \(r\), there is some \(n\) such that \(\BBR^n \xto{r} C\). This notion was introduced by Erd\H{o}s, Graham, Montgomery, Rothschild, Spencer, and Straus \cite{erdos1973}, who showed that every Ramsey set must be \emph{spherical}---that is, lie on the surface of a sphere. They also showed that the product of any two Ramsey sets is itself Ramsey, and thus that every cuboid is Ramsey. Here a \emph{cuboid} means any set of the form \(R=\{0,b_1\}\times \cdots\times \{0,b_s\}\)---in particular, cuboids may have dimension greater than three. The question of which sets are Ramsey remains open.

Of course, if we remove the restriction on the number of colours, then we cannot hope to force a monochromatic copy of any non-trivial configuration, as we could simply colour each point differently. Instead, we ask whether any colouring forces the presence of either a monochromatic configuration or a \emph{rainbow} configuration---one in which every point receives a different colour. For \(n\) a positive integer, and for \(C\subset \BBR^n\), we write
\[\BBR^n \MRto C\]
if any colouring of the points of \(\BBR^n\) admits either a monochromatic copy or a rainbow copy of \(C\). As before, for any \(S\subset \BBR^n\), we write
\[S \MRto C\]
if any colouring of the points of \(S\) admits either a monochromatic copy or a rainbow copy of \(C\). If there is some \(n\) such that \(\BBR^n \MRto C\), we say that \(C\) has the \emph{canonical Ramsey property}, or simply that \(C\) is \emph{canonically Ramsey}.
 
These concepts were first studied by Mao, Ozeki, and Wang \cite{mao2022euclideangallairamseytheory}, who showed that a 30-60-90 triangle is canonically Ramsey. Cheng and Xu \cite{cheng2023euclideangallairamseyvariousconfigurations} showed that right-angled triangles, triangles with circumradius less than their height, squares, and a large family of simplices, are canonically Ramsey. Gehér, Sagdeev, and Tóth \cite{geher2026} proved that all hypercubes are canonically Ramsey---here \emph{hypercubes} are the special case of cuboids with the form \(\{0,b\}^s\). Fang, Ge, Shu, Xu, Xu, and Yang \cite{fang2025canonicalramseytrianglesrectangles} went on to show that all triangles, and indeed all rectangles, are canonically Ramsey. They asked whether all cuboids are canonically Ramsey. Our main theorem answers this question:
\begin{theorem}\label{thm:main}
    Every cuboid is canonically Ramsey.
\end{theorem}

Throughout this paper, we will embed \(\BBR^n\) into \(\BBR^{n'}\), for \(n<n'\), in the usual way. We write \(\mathbf{e}_1,\dots,\mathbf{e}_n\) for the standard basis of \(\BBR^n\).

\section{Overview}
As the construction is quite involved, we give an overview. Unfortunately, even this overview needs a fair amount of notation, to explain the ideas of the proof.

We will need the following slight generalisation of the notation already introduced: for \(S\subset \BBR^n\), and configurations \(C,C'\subset \BBR^n\), we write
\[S\MRto (C,C')\]
if any colouring of the points of \(S\) admits either a monochromatic copy of \(C\), or a rainbow copy of \(C'\).

We will show that the cuboid \(R=\{0,b_1\}\times \cdots\times \{0,b_s\}\) is canonically Ramsey. We begin by constructing a set \(C\) to satisfy the following lemma, which depends on an additional parameter \(a\leq\min\{b_1,\dots,b_s\}\):
\begin{lemma}\label{lem:IR}
    Let \(R=\{0,b_1\}\times \cdots\times \{0,b_s\}\), and suppose that \(a\leq \min\{b_1,\dots,b_s\}\). Then there is a configuration \(C(a)\) such that
    \[C(a) \MRto \left(\{0,a\},R\right),\]
    and \(C(a)\) is Ramsey.
\end{lemma}
To do this, we will first construct simplices \(D_i\) whose vertex sets admit a tree structure such that:---
\begin{itemize}
    \item every vertex has distance \(a\) from each of its siblings, and
    \item every vertex has distance \(b_i\) from all of its ancestors.
\end{itemize}
We will take the underlying tree structure on \(D_i\) to be a complete \(n_i\)-ary tree \(T_i\) of large height \(h\)---that is to say, a rooted tree in which every vertex at distance less than \(h\) from the root has exactly \(n_i\) children---for some very large \(n_i\).
Our configuration \(C(a)\) will be \(D_1\times\cdots\times D_s\), for suitable choices of parameters \(n_i\) and \(h\). Note that as a product of simplices, this configuration is Ramsey.

Fix a colouring of \(C(a)\), and suppose that it admits no monochromatic \(\{0,a\}\).
This induces a colouring on \(T_1\times\dots\times T_s\) such that no two siblings have the same colour.
We say a colouring of a rooted tree is \emph{proper} if no two siblings have the same colour, and additionally each vertex has a different colour from all its ancestors.
We will then use a straightforward counting argument to select a large \(n'_1\)-ary tree \(T'_1\subset T_1\) such that every copy of \(T'_1\) in \(T'_1\times T_2\times \cdots \times T_s\) is properly coloured.
Repeating this operation, we obtain \(T'_i\subset T_i\) such that every copy of any \(T'_i\) in their product \(T'_1\times\cdots\times T'_s\) is properly coloured.

We then choose a long branch \(B_i\) of each \(T'_i\), and note that this corresponds to a regular simplex of side length \(b_i\). Now any two points of \(B_1\times\dots\times B_s\) that differ in only one of the \(s\) coordinates must be differently coloured. We show that we can find arbitrarily large \(B'_i\) such that \(B'_1\times\dots\times B'_s\) is rainbow. In particular, when each \(B'_i\) has size two, this corresponds to a rainbow copy of \(R\).

We then construct a configuration \(S\) such that \(S\MRto (R,R)\) as follows: first choose \(a_1,\dots,a_t\leq \min\{b_1,\dots, b_s\}\) such that \(R'=\{0,a_1\}\times \dots \times \{0,a_t\}\) contains a copy of \(R\). Now write \(S_1=C(a_1)\), and for \(i=2\dots,t\), we choose \(S_i\) so that
\[S_i\xto{m_i}C(a_i),\]
for suitably chosen large \(m_i\). Now we take \[S=S_1\times\dots\times S_t,\]
and consider a colouring of \(S\) with no rainbow copy of \(R\). But now every copy of \(S_1\) must contain a monochromatic copy of \(\{0,a_1\}\). But there are finitely many copies of \(\{0,a_1\}\) in \(S_1\). Hence we may consider an auxiliary colouring of \(S_2\times\dots\times S_t\) corresponding to the locations of these monochromatic copies of \(\{0,a_1\}\). But this is a finite colouring, so for suitable \(m_2\), we can guarantee that each copy of \(S_1\times S_2\) contains a copy of \(\{0,a_1\}\times C(a_2)\) where the two copies of \(C(a_2)\) receive the same colouring.

But now this contains a monochromatic \(\{0,a_2\}\), and hence each copy of \(S_1\times S_2\) contains a monochromatic \(\{0,a_1\}\times \{0,a_2\}\). For suitable \(m_i\), we then repeat this argument and ultimately find a monochromatic copy of \(R\), completing the proof.

\section{Constructing \(C(a)\)}

Let \(T\) be a rooted tree, and let \(a\leq b\) be positive. Then we will construct a configuration \(D(T,a,b)\), in some \(\BBR^n\) for sufficiently large \(n\), together with a bijection \(\phi:T\to D(T,a,b)\), so that
\begin{itemize}
    \item if \(u\) and \(v\) are siblings, then the distance from \(\phi(u)\) to \(\phi(v)\) is \(a\), and
    \item if \(u\) is an ancestor of \(v\), then the distance from \(\phi(u)\) to \(\phi(v)\) is \(b\).
\end{itemize}. We construct \(D(T,a,b)\) and \(\phi\) inductively, as follows:
\begin{itemize}
    \item If \(T\) contains a single vertex, let \(D(T,a,b)\) be a single point, with the obvious bijection.
    \item If \(T'\) is formed from \(T\) by adding \(k\) children to a leaf \(v\) of \(T\), then suppose that \(D(T,a,b)\subset\BBR^n\), and let \(S\subset\BBR^n\) be the points identified by the corresponding \(\phi\) with \(v\) and its ancestors. These all have pairwise distance \(b\), so we may choose \(\mathbf{x}\in \BBR^n\) so that \(\{\mathbf{x}\}\cup S\) is congruent to \(\left\{\mathbf{0},\frac{b}{\sqrt{2}}\mathbf{e}_1,\dots,\frac{b}{\sqrt{2}}\mathbf{e}_{|S|}\right\}\). Now for \(i=1,\dots,k\), define \[\mathbf{w}_i=\mathbf{x}+ \sqrt{\frac{b^2-a^2}{2}}\mathbf{e}_{n+1}+\frac{a}{\sqrt{2}}\mathbf{e}_{n+1+i}.\]
    Then set \(D(T',a,b)=D(T,a,b)\cup \{\mathbf{w}_1,\dots,\mathbf{w}_k\}\), and define the corresponding bijection by extending that of \(D(T,a,b)\) to map the children of \(v\) to \(\{\mathbf{w}_1,\dots,\mathbf{w}_k\}\) in some order.
\end{itemize}

It is easy to verify that our desired distance conditions hold. Note also that by construction, \(D(T,a,b)\) is a simplex, and therefore Ramsey. We will first need a couple of lemmata:

\begin{lemma}\label{lem:propertree}
    Let \(h,n'_1,\dots,n'_s\) be positive integers. Then there exist positive integers \(n_1,\dots,n_s\) such that the following holds: for each \(i=1,\dots,s\), let \(T_i\) be a complete \(n_i\)-ary tree of height \(h\). Suppose that \(T_1\times\dots\times T_s\) is coloured such that in each copy of a \(T_i\), no two siblings have the same colour. Then there are some \(T'_i\subset T_i\) such that each \(T'_i\) is a complete \(n'_i\)-ary tree of height \(h\), and additionally, in each copy of \(T'_i\) within \(T'_1\times\dots\times T'_s\), no vertex and its ancestor have the same colour.
\end{lemma}
\begin{proof}
    We first choose \(n_s,\dots,n_1\) in that order, so that \begin{align*}n_s&\geq n'_s+hn'_1\dots n'_{s-1},\\
    n_{s-1}&\geq n'_{s-1}+hn'_1\dots n'_{s-2}n_s,\\
    &\vdots\\
    n_1&\geq n'_1+hn_2\dots n_s.\end{align*}

    We now choose \(T'_i\) in turn for each \(i=1,\dots,s\) so that each copy of \(T'_i\) in \(T'_1\times\dots T'_i \times T_{i+1} \times \dots \times T_s\) has the desired property. There are \(n'_1\dots n'_{i-1}n_{i+1}\dots n_s\) copies of \(T_i\) in \(T'_1\times\dots T'_{i-1} \times T_i \times \dots \times T_s\). In each of these, every vertex has at most \(h\) children which share a colour with one of their ancestors, since their children are all of different colours.
    
    We say a vertex of \(T_i\) is \emph{bad} if it has the same colour as one of its ancestors in one of these copies of \(T_i\), and \emph{good} otherwise. But now each vertex has at most \(hn'_1\dots n'_{i-1}n_{i+1}\dots n_s\) bad children---so in particular, each vertex has at least \(n'_i\) good children. But now we may choose \(T'_i\) inductively, by including the root, together with \(n'_i\) good children of each vertex included so far that is not a leaf of \(T_i\). This yields a complete \(n'_i\)-ary tree of height \(h\) with the desired property.

    Since \(T'_i\subset T_i\), this implies that that \(T'_1\times\dots\times T'_s\) has the desired property.
\end{proof}

\begin{lemma}\label{lem:rainbowbox}
    Let \(m'_1,\dots,m'_s\) be positive integers. Then there exist positive integers \(m_1,\dots,m_s\) such that the following holds: for each \(i=1,\dots,s\), let \(B_i\) be a set of size \(m'_i\), and let \(B_1\times\dots\times B_s\) be coloured such that any two points which differ in exactly one coordinate have different colours. Then there are sets \(B'_1,\dots,B'_s\) such that each \(B'_i\) has size \(m'_i\), and \(B'_1\times \dots \times B'_s\) is rainbow.
\end{lemma}
\begin{proof}
    We prove this by induction, noting that it is trivial for \(s=1\). Suppose it holds for all \(s\) smaller than the value we are considering, and choose \(m_1,\dots,m_{s-1}\) accordingly. Now set \[m_s\geq {m_1\choose m'_1}\cdots {m_{s-1}\choose m'_{s-1}}(m'_1\cdots m'_{s-1})^2m'_s.\]

    By the inductive hypothesis, each set of the form \(B_1\times\dots\times B_{s-1}\times \{y\}\) contains some rainbow set \(B'_1\times\dots\times B'_{s-1}\times \{y\}\) where each \(B'_i\) has size \(m'_i\). But there are only \({m_1\choose m'_1}\cdots {m_{s-1}\choose m'_{s-1}}\) possible choices for \(B'_1,\dots,B'_s\), so there is some such choice of \(B'_i\) such that \(B'_1\times\dots\times B'_{s-1}\times \{y\}\) is rainbow for at least \((m'_1\cdots m'_{s-1})^2m_s\) choices of \(y\). Let \(B^*_s\) be the set of such \(y\).

    We will choose the elements of \(B'_s=\{y_1,\dots,y_{m'_s}\}\) one at a time. Suppose that we have already chosen \(y_1,\dots,y_\ell\). Now \(B'_1\times\dots\times B'_{s-1}\times \{y_1,\dots,y_\ell\}\) use at most \(\ell m'_1\cdots m'_{s-1}\) different colours. Call these the \emph{bad} colours. Then each copy of \(B^*_s\) in  \(B'_1\times\dots\times B'_{s-1}\times B^*_s\) is rainbow, so has a bad colour in at most \(\ell m'_1\cdots m'_{s-1}\) different positions.
    
    Call \(y\in B^*_s\) \emph{bad} if \(B'_1\times\dots\times B'_{s-1}\times \{y\}\) uses one of these colours, and call it \emph{good} otherwise. But now there are at most \(\ell (m'_1 \cdots m'_{s-1})^2\) bad choices of \(y\in B^*_s\). This is less than \(\left|B^*_s\right|\), so we may choose a \(y_{\ell+1}\) that is good.
    
    But now no colour is reused between different sets of the form \(B'_1\times\dots\times B'_{s-1}\times \{y_i\}\). So since each of these sets is rainbow, the whole set \(B'_1\times\dots\times B'_s\) is, and we obtain the desired result.
\end{proof}

We are now ready to prove Lemma \ref{lem:IR}.

\begin{proof}[Proof of Lemma \ref{lem:IR}]
    Let \(m_1\dots,m_s\) be the values for which Lemma \ref{lem:rainbowbox} holds with parameters \(m'_1,\dots m'_s\) all equal to two Then let \(n_1,\dots,n_s\) be the values for which Lemma \ref{lem:propertree} holds with parameters \(h=\max\{m_1,\dots,m_s\}\), and \(n'_1,\dots,n'_s\) all equal to one.

    Now for each \(i=1,\dots,s\), let \(T_i\) be the complete \(n_i\)-ary tree of height \(h\), and let \(D_i=D(T,a,b_i)\). We now define
    \[C(a)=D_1\times\dots\times D_s.\]
    Notice that since each \(D_i\) is a simplex and therefore Ramsey, this configuration is also Ramsey.

    Now fix some colouring of \(C(a)\), and suppose that it contains no monochromatic copy of \(\{0,a\}\). Then this corresponds to a colouring of \(T_1\times\dots\times T_s\) in which no two siblings in any given copy of a \(T_i\) have the same colour. Thus by Lemma \ref{lem:propertree}, we may choose \(T'_i\subset T_i\) to be complete unary trees of height \(h\) such that in each copy of a \(T'_i\) within \(T'_1\times\dots\times T'_s\), no vertex and its ancestor have the same colour. But these are simply branches of length \(h\), such that each copy of a \(T'_i\) within \(T'_1\times\dots\times T'_s\) is rainbow.

    But now \(T'_1\times\dots\times T'_s\) is coloured in a way that meets the condition of Lemma \ref{lem:rainbowbox}. So for each \(i\), choose \(B_i\subset T'_i\) such that \(\left| B_i \right|=m_i\). Now by Lemma \ref{lem:rainbowbox}, there are some \(B'_i\subset B_i\) each of size two such that \(B'_1\times\dots\times B'_s\) is rainbow.
    
    But now the two points of \(B'_i\) lie on a branch \(T'_i\), and so one is an ancestor of the other. Hence the two points \(X_i\) corresponding to \(B'_i\) within \(D_i\) are at distance \(b_i\). Thus the set \(X_1\times\dots\times X_s\subset C(a)\) is a rainbow copy of \(R\).
\end{proof}

\section{Proof of Theorem \ref{thm:main}}
We have now constructed a configuration which allows us to force either a rainbow \(R\), or a monochromatic interval \(\{0,a\}\). We cannot apply this immediately, however, as this construction is only defined for intervals at most as large as the smallest side length of \(R\). We therefore define a new cuboid, \(R'\), whose side lengths are all at most \(b_{\min} =\min\{b_1,\dots,b_s\}\), and which nevertheless contains a copy of \(R\):
\begin{itemize}
    \item For \(i=1,\dots,s\), set \[t_i=\left\lceil\frac{b_i^2}{b_{\min}^2}\right\rceil.\]
    \item Now for each \(i\), set \[a_{(t_1+\dots+t_{i-1})+1}=\dots=a_{(t_1+\dots+t_{i-1})+t_i}=\frac{b_i}{\sqrt{t_i}},\] and notice that certainly this is at most \(b_{\min}.\)
\end{itemize}
We write \(t=t_1+\dots+t_s\). Now certainly \(R'=\{0,a_1\}\times\dots\times\{0,a_t\}\) contains a monochromatic copy of \(R\), with the form
\[\left\{(0,\dots,0),\left(\frac{b_1}{\sqrt{t_1}},\dots,\frac{b_1}{\sqrt{t_1}}\right)\right\}\times\dots\times\left\{(0,\dots,0),\left(\frac{b_s}{\sqrt{t_s}},\dots,\frac{b_s}{\sqrt{t_s}}\right)\right\}.\]

We are now ready to finish the proof of Theorem \ref{thm:main}.

\begin{proof}[Proof of Theorem \ref{thm:main}]
Since \(R'\) contains a copy of \(R\), it will suffice to construct a configuration \(S\) such that
\[S \MRto (R',R).\]
Recall that each \(a_i\) is at most \(\min\{b_1,\dots,b_s\}\), so the configuration \(C(a_i)\) is defined. And indeed, this configuration is Ramsey. So we may now define \(S\) as follows:
\begin{itemize}
    \item Set \(S_1=C(a_1)\).
    \item For \(i=2,\dots,t\), we define \(S_i\) in turn as follows. Let \(m_i\) be the number of copies of \(\{0,a_1\}\times\dots\times\{0,a_{i-1}\}\) in \(S_1\times\dots\times S_{i-1}\). Then choose \(S_i\) so that
    \[S_i\xto{m_i}C(a_i).\]
    \item Finally, set \(S=S_1\times\dots\times S_t\).
\end{itemize}
We will inductively show that for \(k=1,\dots,t\), we have \[S_1\times\dots\times S_k\MRto\left(\{0,a_1\}\times\dots\times\{0,a_k\},R\right).\]
The base case \(k=1\) is simply Lemma \ref{lem:IR}. We now suppose that this holds for \(k-1\), and prove the statement for \(k\).

Consider a colouring of \(S_1\times\dots\times S_k\), and suppose it has no rainbow copy of \(R\). Certainly this implies that every copy of \(S_1\times\dots\times S_{k-1}\) has a monochromatic copy of \(\{0,a_1\}\times\dots\times\{0,a_{k-1}\}\). This can be in one of \(m_k\) positions within \(S_1\times\dots\times S_{k-1}\). We generate an auxiliary colouring \(\chi\) of \(S_k\) as follows: for each point \(\mathbf{x}\) of \(S_k\), choose \(\chi(\mathbf{x})\) to be some copy of \(\{0,a_1\}\times\dots\times\{0,a_{k-1}\}\) within \(S_1\times\dots\times S_{k-1}\) such that \(\chi(\mathbf{x}) \times \{\mathbf{x}\}\) is monochromatic. Certainly this uses at most \(m_k\) colours.

But then the definition of \(S_k\) implies that there is some copy of \(C(a_k)\) which is monochromatic under \(\chi\)---that is to say, there is some \(C'\subset S_k\) congruent to \(C(a_k)\), and some \(R^*\) congruent to \(\{0,a_1\}\times\dots\times\{0,a_{k-1}\}\), such that for each point \(\mathbf{x}\) of \(C'\), the set \(R^*\times \{\mathbf{x}\}\) is monochromatic. Equivalently, every copy of \(C'\) in \(R^*\times C'\) receives the same colouring.

But now, by the assumption that our colouring contains no rainbow copy of \(R\), Lemma \ref{lem:IR} implies that the colouring of \(C'\) which is common to all copies in \(R^*\times C'\) must contain some monochromatic copy \(I\) of \(\{0,a_k\}\). But now \(R^*\times I\) is a monochromatic copy of \(\{0,a_1\}\times\dots\times\{0,a_k\}\). Hence we have
\[S_1\times\dots\times S_k\MRto\left(\{0,a_1\}\times\dots\times\{0,a_k\},R\right).\]
By induction, this holds for all \(k\). In particular, at \(k=t\), we find that
\[S \MRto\left(R',R\right).\]
Since \(R'\) contains a copy of \(R\), this gives the desired result.
\end{proof}

\section{Further questions}
We have shown that any cuboid is canonically Ramsey. This immediately implies that any subset of a cuboid is canonically Ramsey. In particular, any simplex with \(n\) vertices of the form \(b_1\mathbf{e}_1,\dots b_n\mathbf{e}_n\) is canonically Ramsey.

These sets are all themselves Ramsey. Indeed, Fang, Ge, Shu, Xu, Xu, and Yang \cite{fang2025canonicalramseytrianglesrectangles} conjecture that if a set is Ramsey, then it should also be canonically Ramsey. Remarkably, it also does not seem to be obvious that every canonically Ramsey set is Ramsey, and indeed we have been unable to establish this. We therefore ask:
\begin{question}
    If a set is canonically Ramsey, must it also be Ramsey?
\end{question}

There are configurations which are known not to be canonically Ramsey. Erd\H{o}s, Graham, Montgomery, Rothschild, Spencer, and Straus \cite{erdos1973} gave a spherical colouring of \(\BBR^n\) using three colours which contains no monochromatic copy of \(\{0,1,2,3\}\). Since such a colouring can contain no rainbow copy of this set, it follows that \(\{0,1,2,3\}\) is not canonically Ramsey.

We note that if a set \(S\) is contained in arbitrarily large canonically Ramsey sets, then it must be Ramsey---to see this, when using \(m\) colours, simply choose a canonically Ramsey set \(S'\) with more than \(m\) points that contains a copy of \(S\). Then with only \(m\) colours, we cannot produce a rainbow copy of \(S'\)---so by colouring the same configuration that witnesses that \(S'\) is canonically Ramsey, we must produce a monochromatic copy of \(S\).

Thus any result of the form `If \(S\) and \(T\) are canonically Ramsey, then so is \(S\times T\)' would immediately answer our question in the affirmative. Such a result would be of great interest. Indeed, so would even a weaker result that allows us to generate a new canonically Ramsey set from an old one by adding even a single point.

The simplest unknown special case of this question is the following:
\begin{question}
Is \(\{0,1,2\}\) canonically Ramsey?
\end{question}

If \(\{0,1,2\}\) were canonically Ramsey, this would resolve our first question negatively. Cheng and Xu \cite{cheng2023euclideangallairamseyvariousconfigurations} note that every colouring by distance from the origin, such as the one that witnesses that \(\{0,1,2\}\) is not Ramsey, induces a monochromatic or rainbow copy of this set---so any colouring that witnesses that \(\{0,1,2\}\) is not canonically Ramsey would have a different form.

\section{Acknowledgement}
The author is funded by an Internal Graduate Studentship of Trinity College, Cambridge.

	\bibliographystyle{plain}
\bibliography{main}

\end{document}